\documentclass[a4paper,12pt]{article}
\usepackage{amssymb,amsmath,amsfonts,amsthm}
\usepackage{graphicx,afterpage}
\usepackage{array} 
 
\newcommand{\C}{\ensuremath{\mathcal{C}}}
\newcommand{\sym}{\ensuremath{\mathrm{Sym}}}
\newcommand{\minus}{\ensuremath{\mathrm{Neg}}}

\newcommand{\diam}{\ensuremath{\mathrm{Diam} \; }}
\newcommand{\tbn}{\ensuremath{W(\tilde B_n)}}
\newcommand{\tcn}{\ensuremath{W(\tilde C_n)}}
\newcommand{\tdn}{\ensuremath{W(\tilde D_n)}}
\newcommand{\barb}{\ensuremath{\bar W(\tilde B_n)}}

 \newcommand{\graph}{\ensuremath{\C(G,X)}}

\newcommand{\ep}{\varepsilon}

\newcommand{\zz}{\mathbb{Z}}

\newcommand{\rr}{\mathbb{R}}

\newcommand{\bu}{\mathbf{u}}
\newcommand{\bv}{\mathbf{v}}
\newcommand{\bw}{\mathbf{w}}
\newcommand{\bo}{\mathbf{0}}
\newtheorem{thm}{Theorem}[section]
\newtheorem{lemma}[thm]{Lemma}
\newtheorem{prop}[thm]{Proposition}

\theoremstyle{definition}
\newtheorem{notation}[thm]{Notation}
\newtheorem{defn}[thm]{Definition}

\usepackage
[a4paper,
left=1in,
right=1in,
top=1in,
bottom=1in]
{geometry}
\usepackage{graphicx}
\newcommand{\coverpage}[3]{\thispagestyle{empty}
	\addtocounter{page}{-1}
	\null\vspace*{-1cm} \hfill\includegraphics[scale=1]{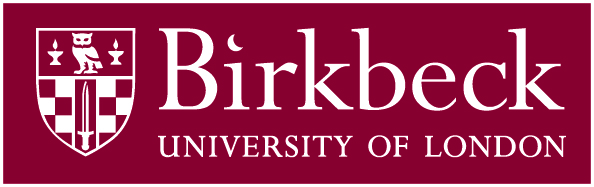} \vskip 2in
	\begin{center} \begin{minipage}{0.7\textwidth}\begin{center}\Huge\bf{#1}\end{center} \end{minipage}\end{center}  \vfill
	\begin{center} {\large By}\bigskip\\ {\large #2}\\ \end{center} \vfill 
	\framebox{\begin{minipage}{\textwidth}
			Birkbeck Mathematical Sciences Preprint Series\hfill
			Preprint Number #3 \\ \\
			\null\hfill www.bbk.ac.uk/ems/research/pure/preprints 
	\end{minipage}}
	\newpage}

\begin{document}
	
	\coverpage{Commuting Involution Graphs in Classical Affine Weyl Groups}{Sarah Hart and Amal Sbeiti Clarke}{42}
	
\title{Commuting Involution Graphs in Classical Affine Weyl Groups}
\author{Sarah Hart and Amal Sbeiti Clarke}
\date{}
\maketitle

\begin{abstract}
In this paper we investigate commuting involution graphs in classical affine Weyl groups. Let $W$ be a classical Weyl group of rank $n$, with $\tilde W$ its corresponding affine Weyl group. Our main result is that if $X$ is a conjugacy class of involutions in $\tilde W$, then the commuting involution graph $\C(\tilde W, X)$ is either disconnected or has diameter at most $n+2$. This bound is known to hold for types $\tilde A_n$ and $\tilde C_n$, so the main work of this paper is to prove the theorem for types $\tilde B_n$ and $\tilde D_n$. 
\end{abstract}

 \section{Introduction}

Let $G$ be a group, and $X$ a subset of $G$. The {\em commuting graph} \graph\ is the graph with vertex set $X$, with an edge joining vertices $x$ and $y$ whenever $x$ and $y$ commute in $G$. If $X$ is a set of involutions, then we call \graph\ a {\em commuting involution graph}. Such graphs have been studied in a wide variety of groups, most often in the cases where $X$ is a conjugacy class of involutions, or the set of all involutions, of $G$. A well-known example of the use of commuting involution graphs is Fischer's work on 3-transposition groups, where in that case $X$ was a conjugacy class of involutions the product of any pair of which had order at most 3.  An early use of commuting graphs, in the case where $X$ is the set of all non-identity elements of $G$, was in Brauer and Fowler's paper on groups of even order \cite{brauerfowler}, where they showed that if $G$ is a group of even order with more than one conjugacy class of involutions, then any two involutions are distance at most 3 apart in \graph. \\

The motivation for the present article is a series of papers by Bates et al., looking at commuting involution graphs for various groups $G$, where $X$ is an involution conjugacy class (\cite{symmetric}, \cite{Finite}, \cite{linear}, \cite{sporadic}). In particular, they showed in \cite{symmetric} that for the symmetric group, \graph\ is either disconnected or has diameter at most 4. Then in \cite{Finite}, they extended this to show that if $G$ is any finite Coxeter group, then \graph\ is either disconnected or has diameter at most 5. Later, Perkins looked at affine groups of type $\tilde A_n$, where it turns out that when \graph\ is connected it has diameter at most 6. Most recently, Hart and Sbeiti Clarke \cite{amal}, considered the case of $\tilde C_n$. Here, \graph, where connected, has diameter at most $n+2$. It is the purpose of this paper to complete the analysis of classical affine groups by dealing with groups of type $\tilde B_n$ and $\tilde D_n$. Our main result is as follows.

\begin{thm} \label{thmclassicalaffine}
Let $W$ be a classical Weyl group of rank $n$ and $\tilde W$ its corresponding affine group, with $X$ a conjugacy class of involutions of $\tilde W$. Then either $\C(\tilde W, X)$ is disconnected, or it is connected with diameter at most $n+2$. 
\end{thm}

This paper is structured as follows. In the remainder of this section we summarise results for finite classical Weyl groups that we will need. In Section 2, we describe the labelled cycle form for affine involutions, first introduced in \cite{amal} for type $\tilde C_n$, which we will use in Section 3 to develop a characterisation of involution conjugacy classes in types $\tilde B_n$ and $\tilde D_n$. In Section 4 we state and prove detailed results on connectedness and diameter of commuting involution graphs in types $\tilde B_n$ and $\tilde D_n$; in particular we obtain necessary and sufficient conditions under which $\C(\tilde W, X)$ is connected, and show that when the $\C(\tilde W, X)$ is connected the diameter is at most $n+2$, thus proving Theorem \ref{thmclassicalaffine} in these cases. Finally, in Section 5 we briefly discuss the exceptional affine groups. \\

For the rest of this section we review the facts about involution conjugacy classes in the classical finite Weyl groups. Throughout this paper we will use the convention that a Coxeter group of type $\Gamma$ will be denoted $W(\Gamma)$, where $\Gamma$ is the associated Coxeter graph. \\

  Let $W$ be of type $A_{n-1}$. Then $W \cong \sym(n)$, and $W$ acts on $\rr_n$ by permuting the subscripts of the standard basis $\{e_1, \ldots, e_n\}$; we view the elements of $W$ as permutations in the usual way.
  If $W$ is of type $B_n$ or $D_n$, then  
        we view the elements of $W$ as signed permutations; they act on $\rr^n$ by permuting the subscripts of the standard basis $\{e_1, \ldots, e_n\}$ and changing their signs. For example given $w =(\stackrel{-}{1}\hspace{0.2cm} \stackrel{+}{2}\hspace{0.2cm}\stackrel{-}{3} ) \in W(B_n)$,
          we have $w(e_1) = -e_2$, $w(e_2) = e_3$ and $w(e_3) = -e_1$.\\
          
        Expressing $\sigma$ as a product of disjoint cycles, we say that a
           cycle $(i_{1}\cdots i_{r})$ of $\sigma$ is {\em positive} if there is
           an even number of minus signs above its elements, and {\em
           negative} if the number of minus signs is odd. For example, $ (\stackrel{+}{1}\hspace{0.2cm} \stackrel{+}{3}\hspace{0.2cm}\stackrel{-}{2} )$    is a negative cycle, whereas $(\stackrel{-}{4}\;\stackrel{-}{5})$ is
           positive. The group $W(B_n)$ consists of all signed permutations of $n$, while $W(D_n)$ is the subgroup of index 2 in $W(B_n)$ consisting of signed permutations with an even number of minus signs. 
        It will be useful to record here for reference the characterisation of conjugacy classes in the groups $W( A_{n-1})$, $W(B_n)$ and $W(D_n)$.

           \begin{thm}\label{1.3.3}
        \begin{trivlist}
               \item[(\textit{i})]Elements of $W( A_{n-1})$  are conjugate if and only if they have the same cycle type. Involution conjugacy classes are parameterised by the number $m$ of transpositions in involutions of the class. 
           \item[(\textit{ii})]  Elements of  $W( B_n)$   are conjugate  if and only if they have the same signed cycle type.  Involutions contain only positive 2-cycles, negative 1-cycles and positive 1-cycles (fixed points). There is exactly one conjugacy class of involutions for each triple $(m,k,l)$ where $m$ is the number of transpositions, $k$ is the number of negative 1-cycles, $l$ is the number of fixed points and $2m+k+l = n$. 
           \item[(\textit{iii})] Conjugacy classes in $W(D_n)$ are parametrised by signed cycle type, with one class for each signed cycle type  except in the case where the signed cycle type contains only even length, positive cycles, where  there are two classes for each signed cycle type. Involutions contain  only positive 2-cycles, negative 1-cycles, and fixed points, with the number of negative 1-cycles always being even. As in $W(B_n)$, there is one conjugacy class for each triple $(m,k,l)$, except for the case $(m,0,0)$ (so that $2m = n)$. In this case, elements are conjugate precisely when the number of minus signs in their expressions as products of signed 
           cycles is congruent modulo $4$.
           \end{trivlist} 
           \end{thm} 
           For example, in $W(D_4)$,  $\stackrel{++}{(1\;2)}\stackrel{++}{(3\;4)}$ is conjugate to 
           $\stackrel{--}{(1\;2)}\stackrel{--}{(3\;4)}$, but not to $\stackrel{--}{(1\;2)}\stackrel{++}{(3\;4)}$.\\

    We write $\diam \mathcal{C}(G,X)$ for the diameter of $\mathcal{C}(G,X)$
            when $\mathcal{C}(G,X)$ is a connected graph, in other words the maximum distance $d(x,y)$ between any $x, y \in X$ in the graph. We state here for reference the known results on connectedness and diameters for finite classical Weyl groups.
               
     \begin{thm}[Theorems 1.1 and 1.2 of \cite{symmetric}]
     Let $W = W(A_{n-1})$ and $X$ be a conjugacy class of involutions having $m$ transpositions and $l$ fixed points (where $2m+l = n \geq 2$).
     \begin{trivlist}
      \item[(\textit{i})] If $l=1$, then $\C(W,X)$ is disconnected.
      \item[(\textit{ii})] If $n=4$ and $m=1$, then $\C(W,X)$ is disconnected.
      \item[(\textit{iii}] If $n \in \{6,8,10\}$ and $l=2$, then  $\diam \C(W,X) =4$.
       \item[(\textit{iv}] In all other cases, $\diam \C(W,X) \leq 3$.
           \end{trivlist}
     \end{thm}
     
         \begin{thm}[Theorem 1.1 of \cite{Finite}] \label{finite}
                             Suppose that $W$ is of type $B_n$ ($n\geq 2$) or $D_n$ ($n \geq 4$), and let $X$ be a conjugacy class of involutions with signed cycle type $(m,k,l)$, where $2m+k+l = n$. Let $t:=\max\{k,l\}$. Then the following hold.
                             \begin{trivlist}
                             \item[(\textit{i})] If $m=0$, then $\C(G,X)$ is a complete graph.
                             \item[(\textit{ii})] If $t = 0$, then $\diam \C(G,X) \leq 2$.
                             \item[(\textit{iii})] If $t = 1$ and $m>0$, then $\C(G,X)$ is disconnected.
                             \item[(\textit{iv})] If $t \geq 2$ and $n>5$, then $\diam \C(G,X) \leq 4$.
                             \item[(\textit{v})] If $n=5$, $m=1$ and $t=2$ then $\diam \C(G,X)=5$. If $n=5$, $m=1$ and $t=3$ then $\diam \C(G,X)=2$. Finally if $n=4$, $m=1$ and $t=2$ then $\C(G,X)$ is disconnected.
                             \end{trivlist}
                           \end{thm}

  \section{Affine involutions and the labelled cycle form}
  
  Our job in this section is to establish what involutions in classical affine Weyl groups look like.
Let $W$ be a finite Weyl group, with root system $\Phi$ contained in a Euclidean vector space $V \cong \rr^n$.
For a given root $\alpha$, recall that the corresponding coroot $\alpha^\vee$ is given by $\alpha^\vee = \frac{2\alpha}{\langle \alpha, \alpha\rangle}$, and we write $\Phi^\vee$ for the set of all coroots.  The affine Weyl group $\tilde W$ is then the semidirect product of $W$ with the translation group of the coroot lattice $L = L(\Phi^\vee)$. We usually identify the translation group with $L$. Then, any  $w \in \tilde W$ can be written in the form $(\sigma,\bv)$ where $\sigma \in W$ and $\bv=(v_1,v_2,...,v_n)\in L$. See, for example, \cite[Chapter 4]{humphreys} for more detail. \\

For $\sigma, \tau \in W$ and $\bu, \bv \in L$  we have $$(\sigma,\bv)(\tau,\bu) = (\sigma\tau,\bv^\tau + \bu).$$
We note that $(\sigma,\bv)^{-1} =(\sigma^{-1}, -\bv^{\sigma^{-1}})$.
Moreover, the element $(\sigma,\mathbf{\bv})$ is conjugate to $(\tau,\mathbf{u})$ via some $(g,\mathbf{w})$ if and only if: \begin{align} (\tau,\mathbf{u}) = (\sigma, \mathbf{v})^{(g,\mathbf{w})}
&= (g^{-1}\sigma g, \mathbf{v}^g + \mathbf{w} - \mathbf{w}^{g^{-1} \sigma g}).\label{eqconj}
\end{align}
   
The reflections of $\tilde W$ are the affine reflections $s_{{\alpha},k}$ ($\alpha \in\Phi$, $k \in\zz$). Recall that, for $\bv$ in $V$,  \begin{align*}
s_{{\alpha}}(\bv) &= \bv - 2\frac{\langle \alpha, \bv\rangle}{\langle \alpha, \alpha\rangle}\alpha = \bv - \langle \alpha, \bv\rangle \alpha^\vee\\
s_{{\alpha},k}(\bv) &= \bv - 2\frac{(\langle \alpha, \bv\rangle - k)}{\langle \alpha, \alpha\rangle}\alpha = s_{\alpha}(\bv) + k\alpha^{\vee}.
\end{align*}If $R$ is a set of  simple reflections for $W$ and $\tilde \alpha$ is the highest root (that is, the root with the highest coefficient sum when expressed as a linear combination of simple roots), then it can be shown that $R\cup\{s_{\tilde\alpha ,1}\}$ is a set of simple reflections for $\tilde W$.\\
          
We will now focus on the classical affine groups. Since it turns out that $W(\tilde A_n) \leq W(\tilde D_n) \leq  W(\tilde B_n) \leq  W(\tilde C_n)$, we assume for the moment that $ \tilde W= W(\tilde C_n)$. 
The Coxeter graph $\tilde C_n$ is as follows:

\begin{center}
	\unitlength 1.00mm
	\linethickness{0.4pt}
	\begin{picture}(100.00,15)(-8,2)
	\put(-25,10){\makebox(0,0)[lc]  {$ \tilde C_n (n\geq 2)$}}
	\put(10,14){\makebox(0,0)[cc]{$r_{1}$}}
	\put(22,14){\makebox(0,0)[cc]{$r_{2}$}}
	\put(34,14){\makebox(0,0)[cc]{$r_{3}$}}
	\put(10,10.5){\line(1,0){12}}
	\put(10,9.5){\line(1,0){12}}
	\put(22,10){\line(1,0){20}}
	\put(62,10){\line(1,0){20}}
	\put(82,10.5){\line(1,0){12}}
	\put(82,9.5){\line(1,0){12}}
	\put(70,10.00){\circle*{2.00}}
	\put(94.00,10.00){\circle*{2.00}}
	\put(10,10.00){\circle*{2.00}}
	\put(22,10.00){\circle*{2.00}}
	\put(34,10.00){\circle*{2.00}}
	\put(82,10.00){\circle*{2.00}}
	\put(70,14){\makebox(0,0)[cc]{$r_{n-1}$}}
	\put(82,14){\makebox(0,0)[cc]{$r_{n}$}}
	\put(94,14){\makebox(0,0)[cc]{$r_{n+1}$}}
	\put(47.00,10.00){\line(1,0){0.4}}
	\put(52.00,10.00){\line(1,0){0.4}}
	\put(57.00,10.00){\line(1,0){0.4}}
	\end{picture}
\end{center}

The root system $\Phi$ of type $C_n$ consists of the long roots $\{2e_i: 1 \leq i \leq n\}$ and the short roots $\{\pm e_i \pm e_j: 1 \leq i < j \leq n\}$. Therefore the set of coroots contains $\{e_1, ..., e_n\}$, meaning that the coroot lattice here is just $\{\bv=(\lambda_1,...,\lambda_n): \lambda_i\in \zz \}$. So $W(\tilde C_n)$ consists of all elements $w = (\sigma,\bv)$ where $\sigma$ is a signed permutation and $\bv \in \zz^n$. For the simple roots of $\Phi$ we can take $2e_1$, $e_1-e_2$, $\ldots$ $e_{n-1} - e_n$, with corresponding reflections $(\overset{-}{1})$, $(\overset{+}{1}\overset{+}{2})$, $\ldots$, $(\overset{+}{n\! -\! 1}\;\;\overset{+}{n})$. The highest root is $2e_n$. Consequently we can set $r_1 = (\overset{\overset{0}{-}}{1})$, $r_i = (\overset{\overset{0}{+}}{i\! -\! 1}\;\,\overset{\overset{0}{+}}{i})$ for $2 \leq i \leq n$, and $r_{n+1} = (\overset{\overset{1}{-}}{n})$.\\

By the definition of group multiplication in $\tilde W $, we see that
the element $(\sigma, \mathbf{v})$ of $\tilde W $ is an involution
precisely when $(\sigma^2, \mathbf{v}^{\sigma} + \mathbf{v}) =(1,\mathbf{0})$. This allows us to characterise the involutions in $\tilde W $.\\

\begin{lemma}[Lemma 2.1 of \cite{amal}] \label{invn}
	A non-identity element $(\sigma, \mathbf{v})$ of $\tilde{W}$ is an involution if and only if $\sigma$, when expressed as a product of disjoint signed cycles, has the form
	\[\sigma =(\overset{+ \;\; + \;\;}{a_1 \, b_1})\cdots 
	(\overset{+ \;\; + \;}{a_t \, b_t}) (\overset{- \;\;\;\;\;\; -}{a_{t+1} \, b_{t+1}}) \cdots (\overset{- \;\;\;\; - \;\;}{a_m \, b_m}) 
	\stackrel{-}{(c_{2m+1})} \cdots \stackrel{-}{(c_{n-l})} (\stackrel{+}{d_{n-l+1}})\cdots\stackrel{+}{(d_n)}\]
	for some $a_i, b_i, c_i, d_i, t, m$ and $l$; and, writing $\bv = (v_1, \ldots, v_n)$, we have   
	$v_{b_i} = -v_{a_i}$ when $\sigma$ contains  $(\overset{+ \;\; +}{a_i \, b_i})$, $v_{b_i} = v_{a_i}$ when $\sigma$ contains $(\overset{- \;\; -}{a_i \, b_i})$ and $v_{d_i} = 0$ for $n-l < i \leq n$.
\end{lemma}

      As described in \cite{amal}, Lemma \ref{invn} allows us to use a shorthand for writing involutions of $\tilde W$.              
    For an element $(w, \bv)$ of $\tilde W$ with $\bv = (v_1, \ldots, v_n)$, then above each signed number $i$ in the expression of $w$ as a product of disjoint signed cycles, we will write $v_i$. For example, we would write $((\overset{-}{1}\hspace{0.1cm} \overset{+}{2}\hspace{0.1cm}\overset{-}{3}), 4e_1 + 5e_2 + 6e_3)$ as 
            $(\overset{\overset{4}{-}}{1}\hspace{0.1cm} \overset{\overset{5}{+}}{2}\hspace{0.1cm}\overset{\overset{6}{-}}{3})$.
           Suppose $(\sigma, \bv)$ is an involution with the form given in Lemma~\ref{invn}. 
            For transpositions $(\overset{\pm \; \pm}{a_i \; b_i})$ of $\sigma$, where the number above $a_i$ determines the number above $b_i$ as described in Lemma~\ref{invn}, we write $(\overset{\stackrel{\lambda}{+}}{a_i b_i})$
             for $(\overset{\stackrel{\lambda}{+} \; \overset{-\lambda}{+}}{a_i \, b_i})$ and $(\overset{\stackrel{\lambda}{-}}{a_i b_i})$
              for $(\overset{\stackrel{\lambda}{-} \; \overset{\lambda}{-}}{a_i \, b_i})$. We will call this the {\em labelled cycle form} of $(\sigma, \bv)$. Where it is helpful, we adopt the convention that cycles $(\overset{\overset{0}{+}}{d_i})$ are omitted, as these fix both $d_i$ and $v_{d_i}$. In labelled cycle form, the simple reflections of $W(\tilde C_n)$ become $r_1 = (\overset{\overset{0}{-}}{1})$, $r_i = (\overset{\overset{0}{+}}{i\! -\! 1\;\, i})$ for $2 \leq i \leq n$, and $r_{n+1} = (\overset{\overset{1}{-}}{n})$.

\begin{notation}\label{notation} For a vector $\bv = (v_1, \ldots, v_n)$ we write $\sum \bv$ for $\sum_{i=1}^n v_i$, the coordinate sum of $\bv$, and $\sum^+ \bv$ for $\sum_{i=1}^n |v_i|$. For a signed permutation $\sigma$, we let $\minus(\sigma)$ be the number of minus signs in $\sigma$. By extension for an element $w = (\sigma, \bv)$ of $\tcn$ we write $\sum w$ for the coefficient sum of $\bv$, $\sum^+ w$ for $\sum^+ \bv$ and $\minus(w)$ for the number of minus signs in $\sigma$. Using the notation above, if $g =(\sigma,\bv)$, we define $f(g) = 2\sum_{i=1}^m v_{a_i} + \sum_{j=1}^{n-2m-l} v_{c_{2m+j}}$. We note that we will only ever be interested in the value $f(g)$ up to congruence modulo 4, so any ambiguity about the choice of ordering of $a_i$ and $b_i$ in transpositions will not matter. \end{notation}
As an example, for the element $$g = \left((\overset{+}{1}\, \overset{+}{2})(\overset{-}{3}\,\overset{-}{4})(\overset{-}{5})(\overset{-}{6})(\overset{+}{7}),(1,-1,3,3,2,4,0)\right) = (\overset{\overset{1}{+}}{1\, 2})(\overset{\overset{3}{-}}{3\, 4})(\overset{\overset{2}{-}}{5})(\overset{\overset{4}{-}}{6})(\overset{\overset{0}{+}}{7}),$$ we have $\sum g = 12$, $f(g) = 14$ and $\minus(g) = 4$.\\

There is a subgroup of index 2 in $W(\tilde C_n)$ which is of type $\tilde B_n$, having the following Coxeter graph.

\begin{center}
	\unitlength 1.00mm
	\linethickness{0.4pt}	\begin{picture}(85,25)(-8,-2)
	\put(10,14){\makebox(0,0)[cc]{$r_{1}$}}
	\put(22,14){\makebox(0,0)[cc]{$r_{2}$}}
	\put(34,14){\makebox(0,0)[cc]{$r_{3}$}}
	\put(10,10.5){\line(1,0){12}}
	\put(10,9.5){\line(1,0){12}}
	\put(22,10){\line(1,0){20}}
	\put(62,10){\line(1,0){8}}
	\put(70,10.00){\circle*{2.00}}
	\put(10,10.00){\circle*{2.00}}
	\put(22,10.00){\circle*{2.00}}
	\put(34,10.00){\circle*{2.00}}
	\put(70,14){\makebox(0,0)[cc]{$r_{n-1}$}}

	\put(47.00,10.00){\line(1,0){0.4}}
	\put(52.00,10.00){\line(1,0){0.4}}
	\put(57.00,10.00){\line(1,0){0.4}}	
	\put(70,10.00){\line(1,-1){12}} \put(70,10){\line(1,1){12}}
	\put(82,22){\circle*{2.00}}
	\put(81,-1){\circle*{2.00}}
	\put(86,22){\makebox(0,0)[cc]{$s$}}
	\put(86,-1){\makebox(0,0)[cc]{$r_{n}$}} \put(-25,10){\makebox(0,0)[lc]  {$\tilde B_n (n\geq 3)$}}
	\end{picture}  
\end{center}

Here $r_1$, $\ldots$, $r_n$ are the same as in the $\tilde C_n$ graph, and we set $s =(\overset{\overset{1}{-}}{n\! -\! 1\;\,n})$. (This agrees with the fact that the $B_n$ root system has short roots $e_i$ and long roots $\pm e_i \pm e_j$, meaning that the highest root is $e_{n-1} + e_n$.) Note that all the generators of this group are elements $(\sigma, \bv)$ satisfying the property that the coefficient sum of $\bv$ is even, and it is not difficult to see that in fact every element with this property lies in $W(\tilde B_n)$. That is, $W(\tilde B_n) = \{(\sigma, \bv) \in W(\tilde C_n) : \sum \bv \in 2\zz\}$, a subgroup of index 2 in $W(\tilde C_n)$.\\

It is also useful to consider another copy of $\tbn$ in $\tcn$, which we will denote $\barb$, whose Coxeter graph is as shown. 

\begin{center}
	\unitlength 1.00mm
	\linethickness{0.4pt}
	\begin{picture}(65,25)(32,-2)
	\put(34,10.00){\line(-1,-1){12}} \put(34,10){\line(-1,1){12}}
\put(22,22){\circle*{2.00}}
\put(22,-1){\circle*{2.00}}
\put(19,22){\makebox(0,0)[cc]{$t$}}
\put(18,-1){\makebox(0,0)[cc]{$r_{2}$}}
	\put(34,14){\makebox(0,0)[cc]{$r_{3}$}}
	\put(34,10){\line(1,0){8}}
	\put(62,10){\line(1,0){20}}
	\put(82,10.5){\line(1,0){12}}
	\put(82,9.5){\line(1,0){12}}
	\put(70,10.00){\circle*{2.00}}
	\put(94.00,10.00){\circle*{2.00}}
	\put(34,10.00){\circle*{2.00}}
	\put(82,10.00){\circle*{2.00}}
	\put(70,14){\makebox(0,0)[cc]{$r_{n-1}$}}
	\put(82,14){\makebox(0,0)[cc]{$r_{n}$}}
	\put(94,14){\makebox(0,0)[cc]{$r_{n+1}$}}
	\put(47.00,10.00){\line(1,0){0.4}}
	\put(52.00,10.00){\line(1,0){0.4}}
	\put(57.00,10.00){\line(1,0){0.4}}
	\end{picture}
\end{center}

In this case $r_2$, $\ldots$, $r_{n+1}$ are as in $W(\tilde C_n)$, and we set $t = (\overset{\overset{1}{-}}{1\;\,2})$. 
It is not hard to see that $\bar W(\tilde B_n) = \{w \in W(\tilde C_n): \sum w \equiv \minus(w) \mod{2}\}$. Certainly the set of such elements has index 2 in $W(C_n)$ and contains $\bar W(\tilde B_n)$; hence it must exactly be $\bar W(\tilde B_n)$. We observe that $\tbn$ and $\barb$ are interchanged by the automorphism of $\tcn$ induced by the nontrivial graph automorphism $\omega$ of its Coxeter graph, which interchanges $r_i$ with $r_{n+2-i}$ for each $i$. \\

The intersection of $W(\tilde B_n)$ and $\bar W(\tilde B_n)$ is another affine Coxeter group, this time of type $W(\tilde D_n)$, with Coxeter graph  

\unitlength 1.00mm
\linethickness{0.4pt}
\begin{picture}(90.00,29)(-40,-4)
\put(22,10.70){\line(-1,1){10}}
\put(22,9.00){\line(-1,-1){10}}
\put(22,10.00){\line(1,0){8}}
\put(50,10.00){\line(1,0){8}}
\put(58,10.00){\line(1,-1){12}}
\put(58,9.02){\line(1,1){12}}
\put(12,20.65){\circle*{2.00}}
\put(70,20.00){\circle*{2.00}}
\put(58.00,10.00){\circle*{2.00}}
\put(12,-1.00){\circle*{2.00}}
\put(22,10.00){\circle*{2.00}}
\put(70,-1.00){\circle*{2.00}}
\put(8,20){\makebox(0,0)[cc]{$t$}}
\put(24,14){\makebox(0,0)[cc]{$r_{3}$}}
\put(58,14){\makebox(0,0)[cc]{$r_{n-1}$}}
\put(74,20){\makebox(0,0)[cc]{$s$}}
\put(75,0){\makebox(0,0)[cc]{$r_{n}$}}
\put(7,0){\makebox(0,0)[cc]{$r_{2}$}}
\put(-25,10){\makebox(0,0)[lc]  {$ \tilde D_n (n\geq 4)$}}
\put(35.00,10.00){\line(1,0){0.4}}
\put(45.00,10.00){\line(1,0){0.4}}
\put(40.00,10.00){\line(1,0){0.4}}
\end{picture}          

where $r_i$, $s$ and $t$ are as defined in $\tbn$ and $\barb$.\\

Finally, for completeness, we note that we can consider $W(\tilde A_{n-1})$ to be the subgroup of $W(\tilde D_n)$ whose elements have no minus signs. 

\section{Conjugacy classes of involutions}

 In this section we give a characterisation of involution conjugacy classes in terms of {\em labelled cycle types}. 
         \begin{defn}
          Let $a$ be an involution in $W(\tilde C_n)$. The {\em labelled cycle type} of $a$ is the tuple $(m,k_e, k_o, l)$, where $m$ is the number of transpositions, $k_e$ is the number of negative 1-cycles with an even number above them, $k_o$ is the number of negative 1-cycles with an odd number above them, and $l$ is the number of positive 1-cycles (fixed points), in the labelled cycle form of $a$. 
          \end{defn}
         For example, the labelled cycle type of ${\stackrel{\stackrel{0}{+}}{(1 2)}}\stackrel{\stackrel{2}{-}}{(3)}\stackrel{\stackrel{1}{-}}{(4)}\stackrel{\stackrel{3}{-}}{(5)} \stackrel{\stackrel{0}{-}}{(6)}
                   \stackrel{\stackrel{3}{-}}{(7)} 
                   \stackrel{\stackrel{0}{+}}{(8)}$ is $(1,2,3,1)$.\\

 \begin{thm}[Theorem 2.6 of \cite{amal}]\label{conjclass}
  Involutions in $W(\tilde C_n)$ are conjugate if and only if they have the same labelled cycle type. In particular, every involution is conjugate to exactly one element $a = a_{m, k_e, k_o, l}$ of the form
 $$a={\overset{\overset{0}{+}}{(1\; 2)}} \cdots \overset{\overset{0}{+}}{(2m-1\;2m)}
  \overset{\overset{0}{-}}{(2m+1)}\cdots \overset{\overset{0}{-}}{(2m+k_e)}
  \overset{\overset{1}{-}}{(2m+k_e+1)}\cdots 
  \overset{\overset{1}{-}}{(n-l)}\overset
  {\overset{0}{+}}{(n-l+1)}\cdots 
  \overset{\overset{0}{+}}{(n)}.$$
  \end{thm}
   
 For completeness, we state the corresponding result for type $\tilde A_n$.
 See \cite{Perkins} for details.
       Since elements of $W(\tilde A_{n-1})$ have no minus signs, involutions here contain transpositions $\stackrel{++}{(a\;b)}$ and fixed points (positive $1$-cycles) only, and the positive $1$-cycles must have zeroes above them. We may therefore omit the signs and $1$-cycles and write involutions as
      \begin{center}
      $\stackrel{\lambda_1}{(a_1\;b_1)}\stackrel{\lambda_2}{(a_2\;b_2)}\cdots \stackrel{\lambda_m}{(a_m\;b_m)}$
      \end{center} 
       \begin{thm}[see \cite{Perkins}] 
       Involutions in $W( \tilde A_{n-1})$ are conjugate if and only if  they have the same number of transpositions,  except in the case where there are no fixed points. In this case there are two conjugacy classes, and elements           
                        $\prod_{i=1}^{m}\stackrel{\lambda_i}{(a_i\;b_i)}$ and $\prod_{i=1}^{m}\stackrel{\mu_i}{(c_i\;d_i)}$                   
         are conjugate if and only if  $\sum_{i=1}^{m} \lambda_i\equiv \sum_{i=1}^{m}\mu_i \mod 2$.                   
          \end{thm}           
        For example, in $W(\tilde A_3)$,   $\stackrel{0}{(1\;2)}\stackrel{0}{(3\;4)}$  and $\stackrel{1}{(1\;2)}\stackrel{1}{(3\;4)}$ are conjugate to each other, but not to $\stackrel{1}{(1\;2)}\stackrel{0}{(3\;4)}$.\\

We now move onto $W(\tilde B_n)$ and $W(\tilde D_n)$. We recall the elementary result from group theory about conjugacy classes in a subgroup $H$ of index 2 in a group $G$. For $x$ in $H$, either $x^H = x^G$, and this happens if and only if $C_G(x)$ contains an element outside of $H$; or $x^G$ splits as two conjugacy classes in $H$, and we have $x^G = x^H \cup (x^g)^H$, where $g$ is any element of $G\setminus H$. In particular this means that if $x^G$ splits in $H$ and $y = x^g$ for some $g$ in $G\setminus H$, then $y$ is not conjugate to $x$ in $H$.  

%

\begin{thm}
	\label{tbnconj}
The involutions of $\tbn$ are those involutions in $\tcn$ with labelled cycle type $(m,k_e,k_o,l)$ where $k_o$ is even. Involution conjugacy classes are parameterised by labelled cycle type, with one conjugacy class for each type except in the case where $k_o = 0$ and $l=0$. In this case there are two conjugacy classes for each type, and two elements $x$, $x'$ with the same labelled cycle type $(m,k_e,0,0)$ are conjugate if and only if $f(x) \equiv f(x') \mod{4}$.   
\end{thm}

\begin{proof}
	Let $x$ be an involution in $\tcn$ with labelled cycle type $(m,k_e,k_o,l)$. Transpositions of $x$ are of the form $(\overset{\overset{\lambda}{+}\; \overset{-\lambda}{+}}{a \;\, b})$ or   $(\overset{\overset{\lambda}{-}\; \overset{\lambda}{-}}{a \; b})$, each of which contributes an even number (zero or $2\lambda$) to $\sum x$. Therefore $\sum x \equiv k_o \mod{2}$, and so $x \in \tbn$ precisely when $k_o$ is even. Suppose this occurs, and that either $l>0$ or $k_o>0$. Then $x$ has a 1-cycle of the form either $(\overset{\overset{0}{+}}{a})$ or $(\overset{\overset{\lambda}{-}}{a})$ for some odd $\lambda$. But then $(\overset{\overset{\lambda}{-}}{a}) \in C_{\tcn}(x) \setminus \tbn$. Therefore $x^{\tbn} = x^{\tcn}$, which in particular means that the conjugacy class of $x$ consists of all involutions with the same labelled cycle type as $x$. It remains to deal with involutions of labelled cycle type $(m,k_e,0,0)$. We will first show that there are two such classes in $\tbn$. Let 
	$\sigma  = (\overset{+}{1\, 2})\cdots (\overset{+}{2m\! -\! 1 \;\, 2m})(\overset{-}{2m+1})\cdots (\overset{-}{2n})$ and $a =  (\sigma, \bo)$. Now $(g,\bw) \in C_{\tcn}(a)$ if and only if $(\sigma,\bo)^{(g,\bw)} = (\sigma,\bo)$. Using Equation \eqref{eqconj} we have
	\begin{align*}
 (\sigma, \bo) &=	(\sigma,\bo)^{(g,\bw)} = (g^{-1}\sigma g, \bo  + \bw - \bw^{g^{-1}\sigma g})\\
 &= (\sigma, (w_1, \ldots, w_n) - (w_1, \ldots, w_n)^\sigma)\\
 &= (\sigma, (w_1 - w_2, w_2 - w_1, \ldots, w_{2m-1} - w_{2m}, w_{2m} - w_{2m-1}, 2w_{2m+1}, \ldots, 2w_{n})).
	\end{align*}
Hence $(g,\bw)$ centralises $a$ if and only if $w_{2i-1} = w_{2i}$ for $1 \leq i \leq m$, and $w_{2m+1} = \cdots = w_n = 0$. Thus $\sum \bw$ is even, which means $C_{\tcn}(a) \leq \tbn$. This means there are two conjugacy classes in $\tbn$ for each labelled cycle type $(m,k_e,0,0)$.\\

	 Suppose $x$ and $x'$ have the same labelled cycle type $(m,k_e,0,0)$. The action of the underlying $W(B_n)$ permutes coordinates and changes signs, but does not change the modulus of coordinates. For example a $\lambda$ over a transposition of $x$ could become $-\lambda$ over a different transposition; a $\mu$ over a 1-cycle could become $-\mu$, but since $\mu$ is even and the calculation of $f(x)$ involves $2\lambda$ rather than $\lambda$, the value of  $f(x^{(g,\bf0)})$ is congruent modulo 4 to $f(x)$ for any $(g,\textbf{0} )\in W(\tilde B_n)$.
Therefore $x$ is conjugate in $\tbn$ to some $\bar x$ given by $ \bar{x}= \prod_{i=1}^m \stackrel{\stackrel{\lambda_i}{+}}{(2i-1\;2i)}\prod_{j=1}^{k_e} \stackrel{\stackrel{\mu_j}{-}}{(2m+j)}$, with $f(\bar x) \equiv f(x) \mod{4}$, and $x'$ is conjugate in $\tbn$ to some $\bar x'$ given by 	$ \bar{x}'= \prod_{i=1}^m \stackrel{\stackrel{\lambda'_i}{+}}{(2i-1\;2i)}\prod_{j=1}^{k_e} \stackrel{\stackrel{\mu'_j}{-}}{(2m+j)}$ with $f(\bar x') \equiv f(x') \mod{4}$. 
It is now sufficient to prove that  $\bar{x}' \in \bar {x}^{W(\tilde B_n)}$ if and only if  $f(\bar {x}')\equiv f(\bar {x}) \mod 4$.\\
		
		Consider the vector $\bw=(w_1,...,w_n)$ given by
		$w_{2i-1}=\lambda'_i$,  $w_{2i}=\lambda_i$,
		$w_{2m+j}=  \frac{1}{2}(\mu'_j-\mu_j)$
		for $1\leq i \leq m$, $1\leq j \leq k_e$. Putting $\bar{x}=(\bar{\sigma}, \bar{\bu})$ we get
		\begin{align*}
		\bar{x}^{(1,\bw)}&= (\bar{\sigma}, \bar{\bu} + \bw-\bw^{\bar{\sigma}})\\
		&= (\bar{\sigma},  \bar{\bu} +(w_1-w_2, w_2-w_1,...,w_{2m-1}-w_{2m},w_{2m}-w_{2m-1},2w_{2m+1},...,2w_n ))\\
		&=( \bar{\sigma}, (\lambda_1,-\lambda_1,\lambda_2,-\lambda_2,...,\lambda_m,-\lambda_m,\mu_1,..., \mu_{k_e})\\  &+(\lambda'_1-\lambda_1,\lambda_1-\lambda'_1,...,\lambda_m,-\lambda'_m,\mu'_1-\mu_1,..., \mu'_{k_e}- \mu_{k_e})\\
		&=(\bar{\sigma}, (\lambda'_1,-\lambda'_1,...,\lambda'_m,-\lambda'_m,\mu'_1,..., \mu'_{k_e}))\\
		&=\bar{x}'.
		\end{align*}
		
		Since there are exactly two conjugacy classes with labelled cycle type $(m,k_e,0,0)$, two elements $x, x'$ with this labelled cycle type will be  conjugate in $\tbn$ if and only if any element of $\tcn$ which conjugates $x$ to $x'$ is contained in $\tbn$.\\
		We have that $\bar{x}'
		\in \bar{x}^{W(\tilde B_n)}$ if and only if $(1,\bw)\in W(\tilde B_n)$ which is if and only if $\sum \bw$ is even. Now 
		\begin{align*}
		\sum_{i=1}^n w_i&= \sum_{i=1}^m  (\lambda_i+\lambda'_i) + \sum_{j=1}^{k_e}\frac{1}{2}(\mu'_j-\mu_j)\\
		& \equiv -\left(\sum_{i=1}^m  \lambda_i+\frac{1}{2}\sum_{j=1}^{k_e}\mu_j\right) + \left(\sum_{i=1}^{m}\lambda'_i+\frac{1}{2}\sum_{j=1}^{k_e}\mu'_j\right) \mod{2}\\
		&= \textstyle\frac{1}{2}\left(f(\bar x') - f(\bar x)\right) \mod{2}.
		\end{align*}
		So  $\bar{x}'
		\in \bar{x}^{W(\tilde B_n)}$  if and only if 
$\frac{1}{2}\left( f(\bar x') - f(\bar x)\right) \equiv 0 \mod{2}$
		which is if and only if $f(\bar x')\equiv f(\bar x)\mod 4$. Thus  ${x}'
		\in {x}^{W(\tilde B_n)}$ if and only if
		$f(x')\equiv f(x)\mod 4$.  
	\end{proof}

\begin{lemma}
	Let $x \in \tcn$ be an involution and let $\omega$ be the automorphism of $\tcn$ corresponding to the non-trivial automorphism of the Coxeter graph for $\tcn$. If the labelled cycle form of $x$ contains $(\cdots \overset{\overset{\lambda}{+}}{i}\cdots)$, then $\omega(x)$ contains $(\cdots \overset{\overset{-\lambda}{+}}{n\! + \! 1 \! - \! i}\cdots)$; if the labelled cycle form of $x$ contains $(\cdots \overset{\overset{\lambda}{-}}{i}\cdots)$, then $\omega(x)$ contains  $(\cdots \overset{\overset{1-\lambda}{-}}{n\! + \! 1 \! - \! i}\cdots)$. In particular, $\omega$ maps labelled cycle type $(m,k_e,k_o,l)$ to labelled cycle type $(m,k_o,k_e,l)$. \label{omega}
\end{lemma}

\begin{proof}
By definition, $\omega(r_i) = r_{n+2-i}$ for $1 \leq i \leq n+1$. Thus $\omega$ interchanges $\overset{\overset{0}{-}}{(1)}$ and $\overset{\overset{1}{-}}{(n)}$, and maps $(\overset{\overset{0}{+}}{i} \; \overset{\overset{0}{+}}{i\! +\! 1})$ to $(\overset{\overset{0}{+}}{n\! +\! 1\! -\! i} \;\; \overset{\overset{0}{+}}{n\! -\! i})$. We quickly see that $\omega$ also maps any $(\overset{\overset{0}{+}}{i} \; \overset{\overset{0}{+}}{j})$ to $(\overset{\overset{0}{+}}{n\! +\! 1\! -\! i} \;\; \overset{\overset{0}{+}}{n\! +\! 1\! -\! j})$ The following table builds up the various images of involution cycles $\sigma$ under $\omega$.

\begin{center}
	\begin{tabular}{c|c|c|c} \hline
		$\sigma$ & Decomposition & Image of Decomposition & $\omega (\sigma)$	\\ \hline
$\overset{\overset{0}{-}}{(n)}$	& $(\overset{\overset{0}{+}}{1 \; n})(\overset{\overset{0}{-}}{1})(\overset{\overset{0}{+}}{1 \; n})$ & $(\overset{\overset{0}{+}}{n \; 1})\overset{\overset{1}{-}}{(n)}(\overset{\overset{0}{+}}{n \; 1})$ & $(\overset{\overset{1}{-}}{1})$ \\
$\overset{\overset{\lambda}{+}}{(i)}$	& $(\overset{\overset{0}{+}}{i \; n})\big[(\overset{\overset{0}{-}}{n})(\overset{\overset{1}{-}}{n})\big]^{\lambda}(\overset{\overset{0}{+}}{i \; n})$ & $(\overset{\overset{0}{+}}{n\! + \! 1 \! - \! i \;\; 1})\big[(\overset{\overset{1}{-}}{1})(\overset{\overset{0}{-}}{1})\big]^{\lambda}(\overset{\overset{0}{+}}{n\! + \! 1 \! - \! i \;\; 1})$ & $(\overset{\overset{-\lambda}{+}}{n\! + \! 1 \! - \! i})$ \\
$\overset{\overset{\lambda}{-}}{(i)}$	& $(\overset{\overset{0}{+}}{i \; n})\big[(\overset{\overset{0}{-}}{n})(\overset{\overset{\lambda}{+}}{n})\big](\overset{\overset{0}{+}}{i \; n})$ & $(\overset{\overset{0}{+}}{n\! + \! 1 \! - \! i \;\; 1})\big[(\overset{\overset{1}{-}}{1})(\overset{\overset{-\lambda}{+}}{1})\big](\overset{\overset{0}{+}}{n\! + \! 1 \! - \! i \;\; 1})$ & $(\overset{\overset{1-\lambda}{-}}{n\! + \! 1 \! - \! i})$ \\
 $(\overset{\overset{\lambda}{+}}{i \; j})$ & $(\overset{\overset{0}{+}}{i \; j})(\overset{\overset{\lambda}{+}}{i})(\overset{\overset{-\lambda}{+}}{j})$ & $(\overset{\overset{0}{+}}{n\! + \! 1 \! - \! i \;\; n\! + \! 1 \! - \! j})(\overset{\overset{-\lambda}{+}}{n\! + \! 1 \! - \! i})(\overset{\overset{\lambda}{+}}{n\! + \! 1 \! - \! j})$ & $(\overset{\overset{-\lambda}{+}}{n\! + \! 1 \! - \! i \;\; n\! + \! 1 \! - \! j})$\\
 $(\overset{\overset{\lambda}{-}}{i \; j})$ & $(\overset{\overset{0}{+}}{i \; j})(\overset{\overset{\lambda}{-}}{i})(\overset{\overset{\lambda}{-}}{j})$ & $(\overset{\overset{0}{+}}{n\! + \! 1 \! - \! i \;\; n\! + \! 1 \! - \! j})(\overset{\overset{1-\lambda}{-}}{n\! + \! 1 \! - \! i})(\overset{\overset{1-\lambda}{-}}{n\! + \! 1 \! - \! j})$ & $(\overset{\overset{1-\lambda}{-}}{n\! + \! 1 \! - \! i \;\; n\! + \! 1 \! - \! j})$\\ \hline
 
\end{tabular}
\end{center}
The result now follows. 
\end{proof}

An immediate consequence of Lemma~\ref{omega} is that for any involution $x$ of $\tcn$, then as long as we do not rearrange the order in which numbers are written in cycles, we have \begin{align} \minus(\omega(x)) &= \minus(x); \nonumber\\ 
\textstyle\sum \omega(x) &= \minus(x) - \textstyle\sum x; \label{omegaprops}\\
f(\omega(x)) &= \minus(x) - f(x).\nonumber\end{align}

As an example, if $x = (\overset{\overset{1}{+}}{1\, 2})(\overset{\overset{3}{-}}{3\, 4})(\overset{\overset{2}{-}}{5})(\overset{\overset{4}{-}}{6})(\overset{\overset{0}{+}}{7})$, then $\omega(x) = (\overset{\overset{-1}{+}}{7\, 6})(\overset{\overset{-2}{-}}{5\, 4})(\overset{\overset{-1}{-}}{3})(\overset{\overset{-3}{-}}{2})(\overset{\overset{0}{+}}{1})$. We have $\minus(x) = 4$,  $\sum x = 12$, $f(x) = 14$,  $\minus(\omega(x)) = 4$, $\sum \omega(x) = -8$ and $f(\omega(x)) = -10$.\\

Lemma \ref{omega} also furnishes a verification of the fact that $\barb = \{y \in \tcn: \minus(y) \equiv \sum y \mod{2}\}$. To see this, we note that $\barb = \omega(\tbn)$. Thus every $y$ in $\barb$ is $\omega(x)$ for some $x$ such that $\sum x \equiv 0 \mod{2}$. Therefore $$\textstyle\sum y = \minus(x) - \textstyle\sum x = \minus(y) - \textstyle\sum x \equiv \minus(y) \mod{2}.$$

\begin{thm}
	\label{barbconj}
	The involutions of $\barb$ are those involutions in $\tcn$ with labelled cycle type $(m,k_e,k_o,l)$ where $k_e$ is even.  Involution conjugacy classes are parameterised by labelled cycle type, with one conjugacy class for each labelled cycle type $(m,k_e,k_o,l)$ except in the case where $k_e = 0$ and $l=0$; in this case there are two conjugacy classes for each type, and two elements $y$ and $y'$ with the same labelled cycle type $(m,0,k_o,0)$ are conjugate if and only if $f(y) + \minus(y) \equiv f(y') + \minus(y') \mod{4}$.   
\end{thm}
\begin{proof} Since $\barb = \omega (\tbn)$, the involutions of $\barb$ are the elements $y=\omega(x)$ where $x$ is an involution of $\tbn$. By Theorem \ref{tbnconj} the involutions of $\tbn$ are precisely the involutions of $\tcn$ with cycle type $(m,k_e,k_o, l)$ where $k_o$ is even. Thus, by Lemma \ref{omega}, the involutions of $\barb$ are precisely those involutions of $\tcn$ with cycle type $(m,k_e,k_o,l)$ where $k_e$ is even. Next we must determine the conjugacy classes. Since $\omega$ maps conjugacy classes to conjugacy classes, Theorem \ref{tbnconj} and Lemma \ref{omega} immediately show that involution conjugacy classes in $\barb$ are parameterised by labelled cycle type with one conjugacy class for each labelled cycle type except where $k_e = 0$ and $l=0$, when there are two classes. Suppose $y$ and $y'$ in $\barb$ have the same labelled cycle type $(m,0,k,0)$. Then $y=\omega(x)$ and $y' = \omega(x')$, for appropriate $x,x' \in \tbn$ of labelled cycle type $(m,k,0,0)$. 
Now $x$ and $x'$ are conjugate in $\tbn$ if and only if $f(x) \equiv f(x') \mod{4}$. By Equation \eqref{omegaprops}, this is if and only if $\minus(x) - f(\omega(x)) \equiv \minus(x') - f(\omega(x'))  \mod{4}$, which is if and only if $\minus(\omega(x)) - f(\omega(x)) \equiv \minus(\omega(x')) - f(\omega(x'))  \mod{4}$. Since $\minus(x)$ and $\minus(x')$ are both even in this case, we get that $y$ is conjugate to $y'$ if and only if 
$f(y) + \minus(y) \equiv f(y') + \minus(y') \mod{4}$.
\end{proof}

\begin{thm}
	\label{tdnconj}
	In $\tdn$ involution conjugacy classes are parameterised by labelled cycle type, with either one, two or four conjugacy classes for each labelled cycle type. Given an element $x$ of $\tdn$ with labelled cycle type $(m,k_e,k_o,l)$, both $k_e$ and $k_o$ must be even, and either $x^{\tdn} = x^{\tbn}$, or $x^{\tdn} = x^{\barb}$ (or both) except in the case where $k_e = k_o = l =  0$ and $n=2m$. Here there are four conjugacy classes and if $x$ and $y$ both have cycle type $(m,0,0,0)$, then $x$ is conjugate to $y$ if and only if both $\minus(x) \equiv \minus(y) \mod{4}$ and $f(x) \equiv f(y) \mod{4}$.  
\end{thm}
\begin{proof}
	Suppose $x \in \tdn$ with labelled cycle type $(m,k_e,k_o,l)$. If any of $k_e$, $k_o$ or $l$ is nonzero, then $x$ has a cycle $(\overset{\overset{0}{+}}{a})$ or $(\overset{\overset{\lambda}{-}}{a})$ for some $\lambda$. But then $(\overset{\overset{\lambda}{-}}{a})$ is contained either in $C_{\tbn}(x) \setminus \tdn$ or $C_{\barb}(x)\setminus \tdn$. Therefore either $x^{\tdn} = x^{\tbn}$ or $x^{\tdn} = x^{\barb}$, or both. The only remaining case is where $m=2n$. If $x$ and $y$ both have cycle type $(m,0,0,0)$ then in order for them to be conjugate in $\tdn$ they must be conjugate both in $\tbn$ and $\barb$, which means that both $f(x) \equiv f(y) \mod{4}$ and  $\minus(x) \equiv \minus(y) \mod{4}$. Hence there are at least four conjugacy classes of this type in $\tdn$, but there are also at most four, because $\tdn$ has index 2 in both $\tbn$ and $\barb$. Hence $x$ is conjugate to $y$ if and only if both $\minus(x) \equiv \minus(y) \mod{4}$ and $f(x) \equiv f(y) \mod{4}$.  
\end{proof}

\section{Connectedness and Diameter Results}

 We begin with an easy observation about  connectedness which allows us in certain cases to infer properties about commuting involution graphs in affine groups from properties of the corresponding graphs in their finite counterparts. For an element $g=( \sigma,\bv)$ in a conjugacy class $X$ of $\tilde W$, we define $ \hat g= \sigma$. Then let $ \hat X$ be the conjugacy class of $\hat g$ in $W$. Clearly if $g,h  \in X$, then $\hat g,\hat h \in \hat X$.     
                    
                    \begin{lemma} \label{easy} Suppose $g, h \in X$. If $d(\hat g,
                    	\hat h) = k$, then $d(g,h) \geq k$. If $\C (W,\hat X)$  is disconnected,
                    	then $\C (\tilde W, X)$ is disconnected.
                    \end{lemma}

                    \begin{proof} The result follows immediately from the observation that if $g$ commutes with $h$ in $G$, then $\hat g$ commutes with $\hat h$ in $W$. 
                    \end{proof}     

The aim of this section is to prove the following result.
\begin{thm}\label{bbdresults} Let $G$ be one of $\tcn$, $\tbn$ or $\tdn$
Let $X$ be a conjugacy class of $G$ with labelled cycle type $(m,k_e,k_o,l)$. Then $\graph$ is disconnected in each of the following cases.
\begin{trivlist}
	\item{(i)} $m=0$ and $l=0$;
	\item{(ii)} $m>0$, $l=0$ and either $k_e = 1$ or $k_o = 1$;
	\item{(iii)} $m>0$ and $\max\{k_e,k_o,l\} = 1$;
	\item{(iv)} $n=4$ and $m=1$;
	\item{(v)} $n=6$, $m=1$, $k_o = k_e = 2$.
\end{trivlist}
In all other cases, $\graph$ is connected with diameter at most $n+2$.
\end{thm}

 The next three results from \cite{amal} give criteria under which labelled 1-cycles, transpositions and double transpositions commute. They
will be used repeatedly in the proofs to follow.
\begin{lemma}[Lemma 2.7 of \cite{amal}]\label{1cycle}
	Let $\alpha \in \{1, \ldots, n\}$ and $\lambda, \mu \in \zz$. Then $\stackrel{\stackrel{\lambda}{-}}{(\alpha)}$ commutes with $\stackrel{\stackrel{0}{+}}{(\alpha)}$ for all $\lambda$, whereas
	$\stackrel{\stackrel{\lambda}{-}}{(\alpha)}$ commutes with  $\stackrel{\stackrel{\mu}{-}}{(\alpha)}$ if and only if $\lambda = \mu$. 
\end{lemma}

\begin{lemma}[Lemma 2.8 of \cite{amal}] \label{2cycle} Let $\alpha, \beta$ be distinct elements of  $\{1, \ldots, n\}$  and let $\lambda$, $\mu$ and $\nu$ be integers. 
	\begin{trivlist}
		\item[(\textit{i})]  $\overset{\overset{\lambda}{+}}{(\alpha\hspace{0.1cm} \beta)}$ and
		$\overset{\overset{\mu}{+}}{(\alpha\hspace{0.1cm} \beta)}$ commute if and only if $\lambda = \mu$, and $\overset{\overset{\lambda}{-}}{(\alpha\hspace{0.1cm} \beta)}$ and
		$\overset{\overset{\mu}{-}}{(\alpha\hspace{0.1cm} \beta)}$ commute if and only if $\lambda = \mu$. But $\overset{\overset{\lambda}{+}}{(\alpha\hspace{0.1cm} \beta)}$ and
		$\overset{\overset{\mu}{-}}{(\alpha\hspace{0.1cm} \beta)}$ commute for all $\lambda$ and $\mu$.
		\item[(\textit{ii})]
		$\overset{\overset{\lambda}{\pm}}{(\alpha\hspace{0.1cm} \beta)}$ and  $\overset{\overset{0}{+}}{(\alpha)}
		\overset{\overset{0}{+}}{(\beta)}$ commute for all $\lambda$, but there is no value of $\mu$ or $\lambda$ for which  $\overset{\overset{\lambda}{\pm}}{(\alpha\hspace{0.1cm} \beta)}$ and  $\overset{\overset{0}{+}}{(\alpha)}
		\overset{\overset{\mu}{-}}{(\beta)}$ or $\overset{\overset{\mu}-}{(\alpha)}
		\overset{\overset{0}{+}}{(\beta)}$ commute.
		\item[(\textit{iii})]
		$\overset{\overset{\lambda}
			{+}}{(\alpha\hspace{0.1cm} \beta)}$ and $\overset{\overset{\mu}{-}}{(\alpha)}
		\overset{\overset{\nu}{-}}{(\beta)}$ commute if and only if $\mu-\nu=2\lambda$, whereas 
		$\overset{\overset{\lambda}{-}}{(\alpha\hspace{0.1cm} \beta)}$ and  $\overset{\overset{\mu}{-}}{(\alpha)}
		\overset{\overset{\nu}{-}}{(\beta)}$ commute if and only if $\mu+\nu=2\lambda$.   
	\end{trivlist}  
	
\end{lemma}

\begin{lemma}[Lemma 2.9 of \cite{amal}]\label{doubletrans} Let 
	$g_1=\overset{\overset{\lambda_1}{+}}{(\alpha\beta)}\overset{\overset{\lambda_2}{+}}{(\gamma\delta)}$, $g_2=\overset{\overset{\lambda_1}{+}}{(\alpha\beta)}\overset{\overset{\lambda_2}
		{-}}{(\gamma\delta)}$, $g_3 = \overset{\overset{\lambda_1}{-}}{(\alpha\beta)}\overset{\overset{\lambda_2}
		{-}}{(\gamma\delta)}$, $h_1=\overset{\overset{\mu_1}{+}}{(\alpha \gamma)}\overset{\overset{\mu_2}{+}}{(\beta\delta)}$, $h_2=\overset{\overset{\mu_1}{+}}{(\alpha \gamma)}\overset{\overset{\mu_2}{-}}{(\beta\delta )}$
	and 
	$h_3=\overset{\overset{\mu_1}{-}}{(\alpha \gamma)}\overset{\overset{\mu_2}{-}}{(\beta\delta )}$,   for distinct $\alpha,
	\beta, \gamma, \delta$ in $\{1, \ldots, n\}$ and integers
	$\lambda_i, \mu_i$. Then
	
	\begin{trivlist}
		\item[(i)] $g_1$ commutes with $h_1$ if and only if $\mu_1 -\lambda_1 =
		\mu_2-\lambda_2$;  \item[(ii)] $g_1$ does not commute with $h_2$; \item[(iii)] $g_1$ commutes with $h_3$ if and only if $\mu_1-\lambda_1 = \mu_2 +\lambda_2$; \item[(iv)]  $g_2$ commutes with $h_2$  if and only if $\mu_1 -\lambda_1 = \mu_2-\lambda_2$; 
		\item[(v)] $g_2$ does not commute with $h_3$;
		\item[(vi)] $g_3$ commutes with $h_3$ if and only if $\mu_1 - \lambda_1 = \lambda_2 - \mu_2$.
	\end{trivlist}
\end{lemma}

\begin{prop}
	\label{n=2m} Let $X$ be a conjugacy class of involutions with labelled cycle type $(m,0,0,0)$ in $G$, where $G$ is either $\tbn$, $\barb$ or $\tdn$. 
	\begin{trivlist}
		\item[(\textit{i})]    
		If $G$ is $\tbn$ or $\barb$, then $\diam \C(G,X) \leq 3$.
		\item[(\textit{ii})] 
		If $G$ is $\tdn$, then $\diam \C(G,X) \leq 4$. 
		
	\end{trivlist}  
\end{prop}   

\begin{proof} Let $x \in X$. Then we have $x= \prod_{i=1}^{m}
	\stackrel{\stackrel{\lambda_i}{+/-}}{(a_i\;b_i)}$ for suitable $\lambda, a_i$ and $b_i$.\\
	       
The two conjugacy classes of labelled cycle type $(m,0,0,0)$ in $\tbn$ are interchanged by conjugation with $(\overset{\overset{1}{-}}{n})$; this induces an isomorphism of their commuting involution graphs. Similarly the corresponding classes of $\barb$ are mapped to those of $\tbn$ by the action of $\omega$, again inducing isomorphisms between the commuting involution graphs. Therefore it is sufficient to consider the case where $x \in \tbn$ has labelled cycle type $(m,0,0,0)$ and $f(x) \equiv 0 \mod{4}$. That is, $X = a^{\tbn}$, where $a = \stackrel{\stackrel{0}{+}}{(1\;2)}\stackrel{\stackrel{0}{+}}{(3\;4)}\cdots \stackrel{\stackrel{0}{+}}{(n-1\;n)}$.
	We can write $x= \prod_{i=1}^{m}
	\stackrel{\stackrel{\lambda_i}{+/-}}{(a_i\;b_i)}$.   $\sum_{i=1}^{m} \lambda_i \equiv   0 \mod{2}$. Now let $y= \prod_{i=1}^{m}
	\stackrel{\stackrel{0}{-/+}}{(a_i\;b_i)} \in X$.     
	By Lemma \ref{2cycle}, $x$ commutes with  $y$, and the fact that $f(y) = 0$ implies that $y \in X$. Now $y$ and $a$ lie in a subgroup isomorphic to $W(B_n)$ and so  $d(y,a)\leq 2$ by Theorem \ref{finite}.     Hence, Diam $C(G,X)\leq 3$. \\

	Now suppose $G =W(\tilde D_n)$. In this case, by Theorem \ref{tdnconj}, $x$ is conjugate to exactly one of the following.       
	
	\begin{align*}     
		w_{1} &=  \stackrel{\stackrel{0}{+}}{(1\;2)}\stackrel{\stackrel{0}{+}}{(3\;4)}\cdots \stackrel{\stackrel{0}{+}}{(n\! - \! 3 \;\, n\! - \! 2)}\stackrel{\stackrel{0}{+}}{(n\! -\! 1 \;\, n)} = tr_4r_6\cdots r_{n-2}r_n  \\   
		w_{2} &= \stackrel{\stackrel{0}{+}}{(1\;2)}\stackrel{\stackrel{0}{+}}{(3\;4)}\cdots \stackrel{\stackrel{0}{+}}{(n\! - \! 3 \;\, n\! - \! 2)}\stackrel{\stackrel{1}{-}}{(n\! -\! 1 \;\, n)} = tr_4r_6\cdots r_{n-2}s \\ 
		w_{3} &= \stackrel{\stackrel{0}{-}}{(1\;2)}\stackrel{\stackrel{0}{+}}{(3\;4)}\cdots \stackrel{\stackrel{0}{+}}{(n\! - \! 3 \;\, n\! - \! 2)}\stackrel{\stackrel{0}{+}}{(n\! -\! 1 \;\, n)}=r_2r_4r_6\cdots r_{n-2}r_n \\
		w_{4}&=  \stackrel{\stackrel{0}{-}}{(1\;2)}\stackrel{\stackrel{0}{+}}{(3\;4)}\cdots\stackrel{\stackrel{0}{+}}{(n\! - \! 3 \;\, n\! - \! 2)} \stackrel{\stackrel{1}{-}}{(n\! -\! 1 \;\, n)}=r_2r_4r_6\cdots r_{n-2}s  
	\end{align*}      
	There are  automorphisms of the Coxeter graph of $\tdn$ mapping  each $w_i$ to any other $w_j$. Therefore, the four conjugacy classes have isomorphic commuting involution graphs. Without
	loss of generality then, assume $x$ is conjugate to $w_{1}$. This occurs if and only if $\minus(x) \equiv 0 \mod{4}$ and $f(x) \equiv 0 \mod{4}$, the latter being equivalent to $\sum_{i=1}^m \lambda_i \equiv 0 \mod{2}$.
	
	If $m$ is even, let $y= \prod_{i=1}^m \overset{\overset{0}{-/+}}{(a_i \, b_i)}$. We have $$\minus(y) \equiv 2m - \minus(x) \equiv \minus(x) \equiv 0 \mod{4}$$ and $f(y) = 0$. Therefore $y \in X$. Now $x$ commutes with $y$, and since $y$ lies in a subgroup isomorphic to $W(D_n)$ we have $d(c,w_{1}) \leq 2$ by Theorem \ref{finite}. Thus $d(x,w_{1}) \leq 3$.
	If $m$ is odd, then as $\sum_{i=1}^m \lambda_i$ is even, there must be at least one even $\lambda_i$. Without loss of generality then, we can suppose $\lambda_1$ is even. Now let $x' = \overset{\overset{\lambda_1}{+/-}}{(a_1 \, b_1)}\overset{\overset{0}{-/+}}{(a_2 \, b_2)}\prod_{i=3}^m \overset{\overset{0}{-/+}}{(a_i \, b_i)}$ and $y' =  \overset{\overset{0}{-/+}}{(a_1 \, b_1)}\overset{\overset{0}{+/-}}{(a_2 \, b_2)}\prod_{i=3}^m \overset{\overset{0}{-/+}}{(a_i \, b_i)}$. We have that $$\minus(x') \equiv 2m-2 - \minus(x) \equiv \minus(x) \equiv 0 \mod{4}$$ and $f(x') = 2\lambda_1 \equiv 0 \mod{4}$. Also note that $\minus(y') = \minus(x')$ and $f(y') = 0$. Thus $x', y' \in X$ and $x'$ commutes with both $x$ and $y'$. Moreover $d(y,w_{1}) \leq 2$. Therefore $d(x,w_{1}) \leq 4$. Hence $\diam\graph \leq 4$.
\end{proof}

In the next proof, and subsequently, we will often work inductively on the rank of the group. To do this, we use a shorthand notation for lower rank subgroups. Let $G = W(\tilde C_n)$. Then, for any subset $\{i_1, \ldots, i_t\}$ of $\{1, \ldots n\}$ there is a subgroup of $G$ isomorphic to $W(\tilde C_t)$, consisting of the labelled signed permutations of $\{i_1, \ldots, i_t\}$. We write $G_{\{i_1, \ldots, i_t\}}$ for this subgroup, and $G_{t}$ for $G_{\{1, \ldots, t\}}$. We use similar notation for $W(\tilde B_n)$ and $W(\tilde D_n)$. 

 \begin{prop}\label{m=1B}  Let $G = W(\tilde B_n)$, and let $x$ be an involution of $G$ with labelled cycle type $(1, k_e, 0, 0)$, where $k_e > 0$. If $k_e \leq 2$, then $\graph$ is disconnected. Otherwise, $\diam\graph \leq n+1$.  \end{prop}

\begin{proof}
	Let $x \in X$. By Theorem \ref{tbnconj}, $x$ is conjugate either to $w_1 = \overset{\overset{1}{-}}{(1 \, 2)}\overset{\overset{0}{-}}{(3)}\cdots \overset{\overset{0}{-}}{(n)}$ or to $w_2 = \overset{\overset{0}{+}}{(1 \, 2)}\overset{\overset{0}{-}}{(3)}\cdots \overset{\overset{0}{-}}{(n)}$. Suppose first that $x$ is conjugate to $w_1$. Then $f(x) \equiv 2 \mod{4}$. 
	
	If $k_e \leq 2$, then the graph is not connected even in the underlying $W(B_n)$, so  \graph\ must be disconnected. So assume $k_e \geq 3$, meaning $n \geq 5$. We will show that if the transposition of $x$ is $\overset{\overset{\lambda}{\ast}}{(1\; 2)}$ for some $\lambda$, then $d(x,a) \leq n+1$. Otherwise $d(x,a) \leq n$. 
	
	We proceed by induction on $n$. The first case to consider is $n=5$, where we have $x=  \overset{\overset{\lambda}{\ast}}{(a \, b)}\overset{\overset{2p}{-}}{(c)}\overset{\overset{2q}{-}}{(d)}\overset{\overset{2r}{-}}{(e)}$ for some integers $\lambda, p, q, r$. Since $f(x) \equiv 2 \mod 4$, we must have $\lambda + p + q + r \equiv 1 \mod{2}$. Conjugation by elements of the centraliser of $w_1$ (which includes the subgroup $\langle \overset{\overset{0}{+}}{(1 \, 2)}, \overset{\overset{0}{+}}{(3\, 4)}, \overset{\overset{0}{+}}{(4 \; 5)}, \overset{\overset{0}{-}}{(5)} \rangle$) does not affect $d(x,w_1)$. Hence we can assume without loss of generality that $x$ contains one of the transpositions $\overset{\overset{\lambda}{+}}{(1\, 5)}$, $\overset{\overset{\lambda}{+}}{(3\; 4)}$ or $\overset{\overset{\lambda}{\ast}}{(1 \, 2)}$. Suppose first that $x$ contains  $\overset{\overset{\lambda}{+}}{(1\, 5)}$. Then $x = \overset{\overset{\lambda}{+}}{(1 \, 5)}\overset{\overset{2p}{-}}{(2)}\overset{\overset{2q}{-}}{(3)}\overset{\overset{2r}{-}}{(4)}$ for some integers $p, q, r$ such that $\lambda + p + q + r \equiv 1 \mod{2}$. Define
	$x_1 = \overset{\overset{1-p-q-r}{-}}{(1 \, 5)}\overset{\overset{2p}{-}}{(2)}\overset{\overset{2q}{-}}{(3)}\overset{\overset{2r}{-}}{(4)}$, $x_2 = \overset{\overset{p+q}{-}}{(2 \, 3)}\overset{\overset{2(1-p-q)}{-}}{(1)}\overset{\overset{2r}{-}}{(4)}\overset{\overset{-2r}{-}}{(5)}$ and $x_3 = \overset{\overset{0}{-}}{(4 \, 5)}\overset{\overset{2(1-p-q)}{-}}{(1)}\;\;\;\overset{\overset{2(p+q)}{-}}{(2)}\overset{\overset{0}{-}}{(3)}$.
	Note that for each $i$ we have $f(x_i) = f(w_1) = 2$, and hence $x_i \in X$. Moreover, by Lemma \ref{2cycle}, $x$, $x_1$, $x_2$, $x_3$, $w_1$ is a path in \graph. Hence $d(x,w_1) \leq 4$. \medskip\\
	If $x$ contains $\overset{\overset{\lambda}{+}}{(3\; 4)}$ , then  $x = \overset{\overset{\lambda}{+}}{(3 \, 4)}\overset{\overset{2p}{-}}{(1)}\overset{\overset{2q}{-}}{(2)}\overset{\overset{2r}{-}}{(5)}$ where again $\lambda + p + q + r \equiv 1 \mod{2}$, and $x$ commutes with $x'  = \overset{\overset{p-r}{+}}{(1 \, 5)}\overset{\overset{2q}{-}}{(2)}\overset{\overset{2\lambda}{-}}{(3)}\overset{\overset{0}{-}}{(4)}$. Now $f(x') = 2(p-r+q+\lambda) \equiv 2(\lambda + p + q + r) = f(x) \mod{4}$. Thus $x' \in X$ and we have seen above that $d(x',w_1) \leq 4$. Consequently $d(x,w_1) \leq 5$.\medskip\\
	Finally, if $x = \overset{\overset{\lambda}{+/-}}{(1 \, 2)}\overset{\overset{2p}{-}}{(3)}\overset{\overset{2q}{-}}{(4)}\overset{\overset{2r}{-}}{(5)}$, then $x$ commutes with $x'' = \overset{\overset{p-q}{+}}{(3 \, 4)}\overset{\overset{2\lambda}{-}}{(1)}\overset{\overset{0}{-}}{(2)}\overset{\overset{2r}{-}}{(5)}$ and $f(x'') \equiv f(x) \mod{4}$, meaning $x'' \in X$. We have seen above that $d(x'', w_1) \leq 5$. Hence $d(x,w_1) \leq 6$. Therefore the inductive hypothesis holds for $n=5$.  
	
	We now assume $n\geq 6$ and proceed inductively. Suppose the transposition of $x$ contains some $a$ with $a > 2$. Then $x$ contains $\overset{\overset{\lambda}{\ast}}{(a \; b)}\overset{\overset{\mu}{-}}{(c)}\overset{\overset{\mu'}{-}}{(d)}$ for some integer $\lambda$ and even integers $\mu$ and $\mu'$.  Then, by Lemma \ref{2cycle}, $x$ commutes with the element $y$ containing  $\overset{\overset{(\mu + \mu')/2}{-}}{(c \;  d)}\overset{\overset{0}{-}}{(a)}\overset{\overset{2\lambda}{-}}{(b)}$, with all its other cycles the same as $x$. Now $f(y) - f(x) = 2|\lambda| + |\mu + \mu'| - (2|\lambda| + |\mu| + |\mu'|)$, which is congruent modulo 4 to $f(x)$ because $\mu$ and $\mu'$ are both even. Thus $y \in X$. Ignoring the cycle $\overset{\overset{0}{-}}{(a)}$ we can work within $G_{\{1,\ldots, a-1, a+1, \ldots, n\}}$, to see that inductively $d(y,w_1) \leq n-1$. Hence $d(x,w_1) \leq n$. If the transposition of $x$ is $\overset{\overset{\lambda}{\ast}}{(1\; 2)}$ then, again by Lemma \ref{2cycle}, $x$ certainly commutes with an element of $X$ which does not have this transposition. So $d(x,w_1) \leq n+1$ as required and $\diam \graph \leq n+1$.\medskip\\
	Finally, the case where $x$ is conjugate to $w_2$ follows from the $w_1$ case. This is because the conjugacy class $X$ of $w_1$ in $W(\tilde B_n)$ is mapped by any element $g$ of $W(\tilde C_n) \setminus W(\tilde B_n)$ to the class $Y$ of $w_2$ in $W(\tilde B_n)$. This map induces an isomorphism of their respective commuting involution graphs. Therefore the graphs have the same diameters.
\end{proof}

\begin{prop} \label{ko=l=0} Let $G = W(\tilde B_n)$. Suppose $X$ is a conjugacy class whose elements have labelled cycle type $(m, k_e, 0, 0)$ where both $m>1$ and $k_e \geq2$, then   \graph\ is connected with diameter at most $n-1$. \end{prop}

\begin{proof} There are two conjugacy classes for each labelled cycle type. Let $x = \prod_{i=1}^m \stackrel{\stackrel{\lambda_i}{*}}{(a_i\;b_i)}\prod_{j=1}^{n-2m} \stackrel{\stackrel{\mu_j}{-}}{(c_j)}$ where the $\mu_j$ are all even.
	Then $x$ is conjugate to exactly one of the following:
	\begin{align*} w_1 &= \stackrel{\stackrel{1}{-}}{(1\;2 )}\stackrel{\stackrel{0}{+}}{(3\;4)}\cdots \stackrel{\stackrel{0}{+}}{(2m-1\;2m )} \; \stackrel{\stackrel{0}{-}}{(2m+1 )}\cdots
	\stackrel{\stackrel{0}{-}}{(n)};\\ 
	w_2 &= \stackrel{\stackrel{0}{+}}{(1\;2)}\stackrel{\stackrel{0}{+}}{(3\;4)}\cdots \stackrel{\stackrel{0}{+}}{(2m-1\;2m)}\; \stackrel{\stackrel{0}{-}}{(2m+1 )}\cdots
	\stackrel{\stackrel{0}{-}}{(n)}. \end{align*}
	In particular, $x$ is conjugate to $w_2$ if and only if $f(x) \equiv 0 \mod{4}$ which is if and only if $2\sum_{i=1}^m \lambda_i + \sum_{j=1}^m \mu_j\equiv 0 \mod{4}$. \\
	
	By assumption, $m \geq 1$ and $k_e \geq 2$. We proceed  
	by induction on $k_e$ to show that $\diam\graph \leq n-1$. Suppose $k_e = 2$. Then $x$ is distance at most 2 from an element $y$ of $X$ which has the transposition $\overset{\overset{0}{+/-}}{(n-1 \;\; n)}$.  To see this, note that if both $n-1$ and $n$ appear in transpositions of $x$, or if both appear in 1-cycles of $x$, then Lemma~\ref{2cycle} or Lemma~\ref{doubletrans}, as appropriate, implies that $x$ commutes with some $x'$ in $X$ which contains a transposition of the form $\overset{\overset{\lambda}{-/+}}{(n-1\;\; n)}$ for some $\lambda$. If $n-1$ and $n$ appear in transpositions of $x$, then we note that for each pair of double transpositions that commute in Lemma~\ref{doubletrans}, the numbers above the first double transposition are $\lambda_1$ and $\lambda_2$, the numbers above the second pair are $\mu_1$ and $\mu_2$, and in every case $\lambda_1 + \lambda_2 \equiv \mu_1 + \mu_2 \mod{2}$, which means $x'$ is guaranteed to be conjugate to $X$. If $n-1$ and $n$ are in 1-cycles of $x$, then $x$ contains cycles of the form $\overset{\overset{\lambda_1}{+/-}}{(\alpha \; \beta)}\overset{\overset{\lambda_2}{+/-}}{(\gamma \; \delta)}\overset{\overset{\mu_1}{-}}{(n-1)}\overset{\overset{\mu_2}{-}}{(n)}$; so, writing $\mu = (\mu_1-\mu_2)/2$, we may choose $x'$ to be $x$ with those cycles replaced with $\overset{\overset{\mu}{+}}{(n-1 \;\; n)}\overset{\overset{\lambda_1 + \lambda_2 - \mu}{-/+}}{(\gamma \; \delta)}\overset{\overset{0}{-}}{(\alpha)}\overset{\overset{2\lambda_1}{-}}{(\beta)}$. In either case, $x'$ is an element of $X$ that commutes with a suitable $y$. The remaining case is when $x$ contains (for example) the 1-cycle $\overset{\overset{\mu}{-}}{(n)}$ and $n-1$ appears in a transposition $\overset{\overset{\sigma}{+/-}}{(\ep\;\; n-1)}$ for some $\ep$ less than $n-1$. Then $x$ commutes with $x''$ in $X$ containing the transpositions $\overset{\overset{\lambda}{-/+}}{(\ep\;\; n-1)}$ and $\overset{\overset{\lambda'}{-/+}}{(\ep'\;\; n)}$ for some $\lambda$, $\lambda'$ and $\ep'$, where we can choose $\lambda$ to ensure that $x'' \in X$. Lemma~\ref{doubletrans} now implies that $x'$ commutes with an appropriate $y$, in particular one containing the transpositions $\overset{\overset{\sigma \pm \lambda}{-}}{(\ep\; \ep')}$ and $\overset{\overset{0}{-}}{(n-1\;\; n)}$. Now $y$ in turn commutes with some $z$ in $X$ with the 1-cycles $\overset{\overset{0}{-}}{(n-1)}$ and $\overset{\overset{0}{-}}{(n)}$. If we ignore these cycles and work in $G_{n-2}$, then a quick check confirms that when $n-2 = 4$ we have $d(z,a) \leq 2$, and when $n-2 > 4$, Proposition \ref{n=2m} 
	tells us that $d(z, a) \leq 3$. Therefore $\diam\graph \leq n-1$. 
	
	Finally, suppose $m\geq 2$ and $k_e > 2$.   Suppose there is some transposition of $x$ containing an element $\alpha$ with $\alpha > 2m$. Then by Lemma \ref{2cycle}(iii) $x$ commutes with some $y \in X$ such that $y$ has the 1-cycle  $\overset{\overset{0}{-}}{(\alpha)}$. By induction $d(y,a) \leq n-2$. Hence $d(x,a) \leq n-1$. The final possibility is that the elements of the transpositions of $x$ are $\{1, 2, \ldots, 2m\}$.  Since $m > 1$ we can use Lemma~\ref{doubletrans} to show that $x$ commutes with some $y$ in $X$ containing the transposition $\overset{\overset{0}{\ast}}{(1\; 2)}$. Working in $G_{\{3, 4, \ldots, n\}}$ (using the case $m=1$ and induction on $m$) we see that $d(y,a) \leq n-2$. Hence $d(x,a) \leq n-1$, which completes the proof of Proposition~\ref{ko=l=0}.  
\end{proof}

We note the following.

\begin{thm}[Theorem 1.2 of \cite{amal}]
\label{cnresults}
Theorem \ref{bbdresults} holds for $\tcn$.
\end{thm}

We next establish that the cases stated to be disconnected in Theorem \ref{bbdresults} are indeed disconnected.

\begin{prop}
\label{disconnect}
Let $G$ be one of $\tbn$, $\barb$ or $\tdn$. Let $X$ be a conjugacy class of involutions in $G$ whose elements have labelled cycle type $(m,k_e,k_o,l)$. Then $\graph$ is disconnected in each of the following cases.
\begin{trivlist}
	\item{(i)} $m=0$ and $l=0$;
	\item{(ii)} $m>0$, $l=0$ and either $k_e = 1$ or $k_o = 1$ (or both);
	\item{(iii)} $m>0$ and $\max\{k_e,k_o,l\} = 1$;
	\item{(iv)} $n=4$ and $m=1$;
	\item{(v)} $n=6$, $m=1$, $k_o = k_e = 2$.
\end{trivlist}
\end{prop}

\begin{proof} 
\begin{trivlist}
	\item{(i)} If $m=0$ and $l=0$, then elements of $X$ consist entirely of negative 1-cycles. So the graph is completely disconnected by Lemma \ref{1cycle}.
	\item{(ii)} Suppose $m>0$, $l=0$ and either $k_e = 1$ or $k_o = 1$. We cannot have $k_e = 1$ and $k_o = 1$ because involutions of $\tbn$ and $\tdn$ require $k_o$ to be even, and involutions of $\barb$ and $\tdn$ require $k_e$ to be even. Since by Lemma \ref{2cycle} 1-cycles can only commute in pairs with 2-cycles, any $x$ in $X$ having the 1-cycle $(\overset{\overset{\lambda}{-}}{i})$ can only commute with elements having the same 1-cycle. So $\graph$ is disconnected.
	\item{(iii)} Suppose $m>0$ and $\max\{k_e,k_o,l\} = 1$. If $l=0$, then $\graph$ is disconnected by (ii).  If $l > 0$ then $X = X^{\tcn}$ by Theorems \ref{tbnconj}, \ref{barbconj} and \ref{tdnconj}, and therefore  $\graph$ is disconnected by Theorem \ref{cnresults}. 
	\item{(iv)} If $n=4$ and $m=1$ then $\graph$ is disconnected even in the underlying $W(B_n)$ or $W(D_n)$, so $\graph$ is disconnected.
	\item{(v)} If $n=6$, $m=1$, $k_o = k_e = 2$, then $X = X^{\tcn}$ by Theorems \ref{tbnconj}, \ref{barbconj} and \ref{tdnconj}, and therefore  $\graph$ is disconnected by Theorem \ref{cnresults}.\qedhere
\end{trivlist}
\end{proof}

\paragraph{Proof of Theorem \ref{bbdresults}}

By Theorem \ref{cnresults}, it only remains to prove Theorem \ref{bbdresults} for $\tbn$ and $\tdn$. First let $G = \tbn$ and let $X$ be a conjugacy class of involutions with labelled cycle type $(m,k_e,k_o,l)$. In this group we must have $k_o$ even. If $l>0$ or $k_o > 0$, then $X = X^{\tcn}$, and so the results from $\tcn$ on connectedness and diameter apply; in particular all cases not given in Proposition \ref{disconnect} are connected with the diameter is at most $n+2$. It remains to deal with the case where $l=0$ and $k_o = 0$. Here, if $k_e = 0$, then $\diam \graph \leq 3$ by Lemma \ref{n=2m}. If $k_e = 1$, then $\graph$ is disconnected by Proposition \ref{disconnect}. So suppose $k_e \geq 2$. If $m=0$ then $\graph$ is completely disconnected. If $m=1$, then by Proposition \ref{m=1B} either $k_e \leq 2$ and $\graph$ is disconnected, or $k_e > 2$ and $\diam \graph \leq n+1$. If $m > 1$, then by Proposition \ref{ko=l=0}, we have that $\graph$ is connected with diameter at most $n-1$. Therefore Theorem \ref{bbdresults} holds for $\tbn$. Since $\omega$ maps classes of $\tbn$ to classes of $\barb$, it follows immediately that Theorem \ref{bbdresults} holds also for $\barb$. Now suppose $G = \tdn$. By Theorem \ref{tdnconj}, if any of $k_o, k_e$ or $l$ are nonzero, then either $X = X^{\tbn}$ or $X = X^{\barb}$ or both. Therefore the results from $\tbn$ and $\barb$ apply. The only remaining case is when $X$ has labelled cycle type $(m,0,0,0)$. Here, $\graph$ is connected with diameter at most 4, by Proposition \ref{n=2m}. Therefore Theorem \ref{bbdresults} also holds for $\tdn$.\qed

\section{Conclusion}

A natural question is whether the $n+2$ bound continues to hold for the exceptional affine groups. There are some preliminary results in \cite{amalthesis} and \cite{noteaffine}. In particular, complete results are known for types $\tilde F_4$ and $\tilde G_2$, along with results for some classes in $\tilde E_6$, $\tilde E_7$ and $\tilde E_8$. Full information for types $\tilde E_6$, $\tilde E_7$ and $\tilde E_8$ requires more work and will be the subject of a forthcoming paper by the current authors.

 \end{document}